\documentclass[]{article}
\usepackage{graphicx,enumerate}

\usepackage{jr,comment}
\usepackage[ruled,vlined]{algorithm2e} 
\usepackage{url}

\newtheorem{theorem}{Theorem}

\newtheorem{lemma}{Lemma}
\newtheorem{Assumption}{Assumption}

\newtheorem{example}{Example}
\def\bb#1{\mathbb{#1}}

\def\en{\infty}
\def\summ#1#2#3{\sum_{#1=#2}^{#3}}

\def\boldsymbol#1{\setbox\ewb\hbox{$#1$}%
    \setlength{\deno}{-\wd\ewb+0.05em}{ #1}\hspace{\deno}{#1}}
\begin{document}

%\baselineskip=23pt
%begin{frontmatter}

\title{Empirical Bayes Estimation  of \\the Mean of a Function of the Latent Variable  \\with Applications to the
Treatment of Nonresponse}
\pagestyle{myheadings}
\markright{SPMLE with nonresponse}
\author{  Eitan Greenshtein\footnote{Central Bureau of Statistics, Israel, \texttt{eitan.greenshtein@gmail.com}} and \jr\footnote{University of Michigan, USA, \texttt{yritov@umich.edu}}}

%eitan.greenshtein@gmail.com

\maketitle
{\bf Abstract}

We consider the estimation of linear functionals of the mixing distribution in a nonparametric empirical Bayes framework. Our main interest is in situations in which the mixing distribution is only partially identifiable, as may arise in complex sampling situations with nonresponse. We argue that estimating the functional by applying it to the semiparametric maximum likelihood estimator of the mixing distribution is an efficient tool, even when the maximum likelihood estimator is not unique. 

\noindent\emph{Some key words:} Empirical Bayes; G-modeling; semiparametric maximum likelihood. 
\begin{comment}
We consider a nonparametric empirical Bayesian framework. Let  $Y_i$ be random variables, $Y_i \sim f(y\mid \theta_i)$, $i=1,\dots,n$, where
$\theta_i \sim G$, and $\theta_i \in \Theta$ are independent. The observed variable $Y_1,\dots,Y_n $ are conditionally  independent
given the latent $\theta_1,\dots,\theta_n$.
The mixing distribution  $G$ is unknown and is assumed to belong  to a nonparametric class
 $\scg$.

Let $\eta(\theta)$  be a function of $\theta$. We address the problem of consistently estimating  $\eta_G\equiv \E_G \eta(\theta)  $. This problem becomes
particularly challenging when $G$ cannot be consistently estimated from the observed data.
We investigate the class of functions that can be identified. We argue that the semiparametric maximum likelihood estimator of $G$, although not consistent, is an efficient method to estimate parameters such as $\eta_G$. We suggest interval estimators, again based on the MLE of $G$ in partially identified situations. 

Our main motivation is estimation in the contexts involving nonresponse and missing data.
\end{comment}

\section{Introduction}

Our goal in this paper is to discuss the feasibility of the estimation of parameters of the latent distribution in a context of the nonparametric empirical Bayes (NPEB) model when this distribution is unidentifiable, and, in particular, under nonresponse. 
Let $\theta_1,\dots,\theta_n$ be unobserved i.i.d. from a distribution $G\in\scg$. Conditional on $\theta_1,\dots,\theta_n$, the random variables $Y_1,\dots,Y_n$ are independent, where $Y_i$ is sampled from a density with respect to dominating measure $\mu$  which depends only on $\theta_i$,  $Y_i \sim f(\cdot\mid \theta_i)$, $i=1,\dots,n$. Finally,
the parameter of interest is $\eta_{G}= \E_G \eta(\theta).$
Given an estimator $\hat{G}$ for $G$, a natural  estimator  for $\eta_G$,  is:
$ \eta_{\hat{G}}= \E_{\hat{G}} \eta(\theta)$.

We consider situations in which the $Y_i$s are only partially observed---for example, they may be truncated or censored. Let  $t(Y_1),\dots,t(Y_n)$ be the actual observations. The likelihood function is well defined since the family of distributions of $t(Y)$ is dominated, and therefore also its maximization with respect to $G$. We concentrate, therefore, on the semiparametric maximum likelihood estimator (SPMLE) of $G$ (see the formal definition in the sequel).   To simplify notations, we will often write $Y_i$ when there is no confusion.

 In what follows, consistency  of  a sequence of estimators $\hat{G}^n$ for $G$ is understood in the sense of  almost sure (a.s.) weak convergence;
we write $\hat{G}^n \Rightarrow G$.
In many cases, no consistent sequence  of estimators  $\hat{G}^n$ for $G$ exists, and the task of consistently estimating
$\eta_G$ appears infeasible.
 
However, if there were a sequence of estimators $\hat{G}^n$ converging weakly to $G$ almost surely, then the
corresponding sequence
of estimators $\eta_{\hat{G}}$ would be consistent for $\eta_G$, provided $\eta(\theta)$ is continuous
and bounded. Specifically, $\hat{G}^n \Rightarrow G$, implies
 $\eta_{\hat{G}^n} \rightarrow_{a.s.} \eta_G$.

%{\bf Remark.} In fact,  results about weak convergence of SPMLE $\hat{G}%^n$ to $G$ are only true  as almost sure convergence. Hence, the implied %convergence  of $\eta_{\hat{G}}$ to $\eta_G$, is convergence with %probability 1.
%For simplicity, we will slightly  abuse the statements  and drop the almost %sure
%statement, in most of the remaining of the paper.

In the following trivial example, we demonstrate the difficulty when there is no definite sequence $\hat{G}^n$ that converges weakly to $G$, and that, for at least some parameters, this may not be a real problem.

\begin{example} \label{ex:first}
Let  $\theta_i \sim G$ be i.i.d, $Y_i \sim B(1,\theta_i)$, $\theta_i \in [0,1]$, $i=1,\dots,n$.  Let $\bar Y_n\eqdef n^{-1}\summ i1n Y_i$. It can be easily shown that $\hat G$ is an SPMLE if, and only if,  $\E_{\hat G}\theta=\bar Y_n$.  Thus, if $\eta(\theta)=\theta$, then $\E_{\hat G}\eta=\bar Y_n$, and $\eta_{\hat G}$ is consistent. On the other hand, if  $\eta(\theta)\not\equiv\theta$, e.g., $\eta(\theta)=\theta^2$, $\eta_G$ is unidentified, and $\eta_{\hat G}$ depends on the specific $\hat G$. 

\end{example}

A key difference between the different functions mentioned in the example, is that
in the case $\eta(\theta)=\theta$ we have  $\eta(\theta)= \E_{\theta} h(Y)$ for $h(Y)=Y$.
Hence  $\E_G\eta(\theta)=\E_{G}  h(Y)$, in which case a natural unbiased estimator of $\eta_G$
is  
\begin{equation}\label{eqn:missing}
\hat\eta_G = \frac{1}{n} \sum_i h(Y_i).  
\end{equation}
Thus, in situations where $\eta_G=\E_G h(Y)$, we have the above natural estimator, which does not involve
estimating $G$. Greenshtein and Ritov (2022)  show that the above estimator and the estimator $\eta_{\hat{G}}$ are identical
in a more general setup.

The main focus of this paper is the estimation of $\E_G \eta(\theta)=\E_{G}  h(Y)$ in situations where $Y_i$ may be only partially observed, or even when there are missing not at random (MNAR) observations among the $Y_i$s.
 In such cases, the naive estimator \eqref{eqn:missing}, applied only to the complete observations, 
may be inconsistent.
We  demonstrate in this paper that, similarly to the situation in Example \ref{ex:first}, estimators of the form $\E_{\hat{G}}  \eta(\theta) = \eta_{\hat{G}}$ often
remain consistent, even when there are MNAR observations $Y_i$ and $G$ cannot be consistently estimated.

The problem of estimating  $\eta_G$, the mean of $\eta(\theta)$ under the mixture $G$, has scarcely been studied in the 
empirical Bayes literature. It has been examined in 
Greenshtein and Itskov (2018) and in Greenshtein and  Ritov (2022). In those papers,
we were intrigued by the surprisingly strong performance of 
the estimator $\eta_{\hat{G}}$ in situations where $G$ cannot be consistently estimated and is 
nonidentifiable. Understanding this phenomenon motivates the present study.

A related line of research exists in econometrics;  see, e.g.,  Dobronyi, Gu, and Kim (2021). In their work, the goal is to estimate functionals of a latent
distribution of latent variables (analogous to our mixture distribution of the parameters, 
which are `latent variables'), in a dynamic panel logit model. Their approach involves estimating moments of the latent distribution, and leveraging the relationship between moments of a distribution and the distribution itself (the moment problem). They also report surprisingly good performance of estimates of
functionals of the unknown distribution of the latent variables, in cases where that distribution is nonidentifiable from the observed data.

In cases with  MNAR observations, the available data are typically not representative of the population due to a selection bias. This phenomenon also occurs in observational studies.
The NPEB approach in the current paper is designed to correct such bias. The papers by Eckles, Ignatiadis, et al.  (2025) and by Robbins and Zhang (1991) (see also references therein) apply empirical Bayes ideas to handle selection bias due to ``Regression Discontinuity Designs". As an example, consider a clinical trial where it is desired to estimate the treatment effect of a new drug. However, suppose that only ``difficult cases" are assigned to the treatment group.

The paper of Vardi (1985) has a similar motivation of handling selection bias; see also Turnbull (1976).
In both of those papers the goal is to estimate the marginal distribution of $Y$.
The former handles selection bias in a general context; the latter is motivated by interval censoring in survival
analysis. Both essentially use SPMLE. We consider a selection bias that is due to nonresponse. The parametrization and setup in
the present paper differ substantially.

\section{Preliminaries }
\subsection{Basic definitions}

Given a mixing  distribution $G$ and a dominated family of distributions with densities {\nolinebreak$\{ f(\cdot\mid \theta):\; \theta \in \Theta \}$} with respect to some dominating measure $\mu$, define
$ f_G(y)= \int f(y\mid \vartheta) dG(\vartheta).$

 Given the observations  $Y_1,\dots,Y_n$,
the semiparametric maximum likelihood estimator (SPMLE) $\hat{G}$ of $G$  is any $\hat{G}$ satisfying
$\hat G = \argmax_{G \in \scg} \; \Pi_{i=1}^n f_{{G}}(Y_i)$. Since the distribution of $Y$ is dominated, we do not need the Kiefer and Wolfowitz (1956) device of the nonparametric MLE. 

Under the above setup, we say that the model is identifiable if for every
$G_1,G_2 \in \scg$, $f_{G_1}=f_{G_2}\text{ a.e. } \Rightarrow \; G_1=G_2.$ A parameter $\eta: \scg\to \R^p$ is identifiable if $G_1,G_i\in\scg$, $f_{G_1}=f_{G_2}\Rightarrow \eta_{ G_1}=\eta_{G_2}$. Clearly, a model can be unidentifiable while a specific parameter of the model is identifiable. In that case, we may refer to the model as a partially identifiable model. 

Revisiting  Example \ref{ex:first}, in light of these definitions, it may be seen that in this 
example, any two distributions $G_1$ and $ G_2$  with the same mean cannot be distinguished. Consequently, regardless of the 
sample size, the true distribution $G$ cannot be consistently estimated. However, $\eta(G)=\E_G(p)$ is identifiable.

\subsection{Empirical Bayes: a brief review.}
The idea of empirical Bayes was suggested by Robbins (see  Robbins 1953, 1956, 1964).
In parametric empirical Bayes, the prior (or mixing distribution) $G$ is assumed to belong to a parametric family of distributions, and the task
is to estimate $G$ based on the observed $Y_1,\dots, Y_n$.
For example, consider the case  $$f(y\mid \theta)=N(\theta,1), \; \theta \in \Theta \subseteq \R,$$ 
 with $\scg\ =\{N(0,\sigma^2), \; \sigma^2 \in \R_+ \}$, where $\sigma^2$ is unknown. 

 In nonparametric 
empirical Bayes, the prior $G$ is allowed to be any distribution on the parameter set $\Theta$.
Two main approaches exist: 
\begin{enumerate}
  \item   {\bf  $f$-modelling approach:} This approach estimates $f_{G}(y)$-the marginal density of $Y$, directly without estimating $G$.

  \item     {\bf   $G$-modelling  approach:} This approach estimates $G$ directly. 

\end{enumerate}

See Efron (2014) for 
further discussion of these two approaches. 
 For reviews on empirical Bayes and its applications see
 e.g.,  Efron (2010),   
Zhang (2003), and Koenker and Gu (2026).

In nonparametric  $G$-modelling, $G$ is typically estimated using SPMLE. Traditionally, 
the computation of  SPMLE was carried out via the EM algorithm, see Laird (1978). More recently, Koenker and Mizera (2014) proposed computing using convex optimization techniques.

%{\bf Estimating the individual parameters $\boldsymbol{\theta}_i$.}
The more common task in empirical Bayes is the point estimation of the individual parameters $\theta_i$,  where $Y_i \sim f(y\mid \theta_i)$, $i=1,\dots,n$. However, in this paper, we emphasize the estimation of the mean of 
various functions $\eta(\theta)$ under the mixture $G$. Indeed, the more common task is problematic without an identifiability assumption. A natural estimator of the  individual parameters $\theta_i$, $i=1,\dots,n$, based on SPMLE $\hat{G}$, is through:
$\hat{\theta}_i= \E_{\hat{G}} (\theta_i\mid  Y_i).$
The difficulty under non-unique and inconsistent $\hat{G}$ is illustrated with Example \ref{ex:first}. Clearly, $\E_{\hat G}(\theta\mid Y{=}1)=\int \theta^2 d\hat G(\theta)/\int \theta d\hat G(\theta)=\int \theta^2 d\hat G(\theta)/\bar Y_n$, which depends on the specific $\hat G$, since $\hat G$ is only restricted by $\int\theta d\hat G(\theta)=\bar Y_n$. The above ambiguity remains regardless of how large  $n$ is.

It is worth noting that  Greenshtein and Ritov (2022), 
show the following identity: 
$$\eta_{\hat{G}}=\frac{1}{n} \sum_j \E_{\hat{G}}(\eta(\theta_j)\mid Y_j),$$ holds for any SPMLE $\hat{G}$.

\subsection{Overview and further preliminaries} \label{sec:over}

To illustrate the key ideas, consider a simple scenario. Suppose the goal is to estimate the proportion
of unemployed individuals,
based on a sample of size $n$. 
Let $Y_i=1$ if individual $i$ is unemployed, $Y_i=0$ otherwise. 
The obvious estimator is $ \sum Y_i/n$. However, this estimator becomes problematic if there are some nonresponse cases (or missing values), particularly if individuals that did not respond are MNAR. For instance, unemployed individuals are more likely not to respond.

Another example, suppose that $Y_i \sim f(y\mid \theta_i)$.  Let $I_i$ be the indicator of the event:
``a response was obtained from individual $i$", let $h(Y_i)$ denote, for example,  the salary associated with 
individual  $i$.  
The goal is to estimate 
$$\E_{G} h(Y) =\E_G E(h(Y)\mid  \theta)=\E_G \eta(\theta) \equiv \eta_G,$$
where $ \eta(\theta)=E(h(Y)\mid \theta).$
%Defining $\eta(\theta)=\E_G(h(Y)\mid \theta)$ we obtain $\E_G h(Y)= \E_G \eta(\theta)\equiv \eta_G$.
When there is a positive correlation, under $G$, between $I$ and $h(Y)$, a positive bias
of the naive estimator is expected.
Specifically, suppose $\sum I_i>0$; then  the naive estimator, ${  \sum_i h(Y_i) I_i}/{\sum I_i},$
is clearly biased for $\E_G h(Y)$.

Neutralizing the correlation when conditioning on $\theta_i$ aligns with the principle of missing at random (MAR),
as discussed in Little and Rubin  (2002), where conditioning is performed on known strata. Our situation is more complex since the ``strata" (parameters)
are unknown and unobserved.

The current paper deals in particular with sampling situations where nonresponse is considerable and cannot be ignored. In particular, we argue:
\begin{enumerate}
  \item In Lemma \ref{lem:main}, Section \ref{sec:uniqueness}, and Lemma \ref{lem:mainTruncated}, Section \ref{sec:trunc}, we show that the SPMLE of the mixing distribution, even if it is not unique as an estimator of the distribution, is well-defined as an estimator of the observed distribution. 
  \item In Section \ref{sec:spmleConsistent}, we also discuss its convergence rate in total variation.
  \item Assumption \ref{ass:hfun} and Theorem \ref{th:funct}, Section \ref{sec:ident}, characterize the functionals that can be estimated as a limit of the observed density at a finite number of points.
  \item We, thus, argue that even when the SPMLE is not unique, and the mixing distribution is unidentifiable, the functionals of the SPMLE are good estimators for identifiable functionals. 
  \item In Section \ref{sec:CI}, we discuss the partially identified situation and the use of the SPMLE in such situations. We suggest CIs that include the unidentifiable range.
  \item We exemplify the concepts discussed in this paper in different sampling examples. See Sections \ref{sec:ex} and \ref{sec:ident}. A simulation study is given in Section \ref{sec:CI}, Example \ref{ex:sim}, and real data is analyzed in Example \ref{ex:data}.
  
\end{enumerate}

\section{Examples.} \label{sec:ex}

The next example was considered in Greenshtein and Itskov (2014).

\begin{example} \label{ex:GI}
Consider a survey in which each sampled item is approached up to $K$
times. If no response is obtained within $\kappa$ attempts, the outcome is recorded as ``nonresponse." The  response  $X$
takes the values, $X=x_1,\dots,x_S$, e.g., $x_1=\mbox{"employed"}, \; x_2=\mbox{ "partially employed"}, \; x_3=\mbox{ "unemployed"}$, as in Greenshtein and Itskov (2014). 
Let $K_i$  be the number of attempts until a response; let $Y_i= (X_i, K_i)$. which is observed when $K_i\le \kappa$. Thus,  we observe
$t(Y_i)$, where $t(Y_i)=(X_i,K_i)$  if  $K_i \leq K$ and   $t(Y_i)=\mbox{\rm``nonresponse"}$  otherwise.
 leading to
$\kappa S +1$ possible  observed values.

We assume that the observed number of attempts $K_i$ 
is a truncated geometric random variable, with parameter $\pi_i$. The probability of $X_i=x_s$ is $p_{is}$, $ s=1,2,\dots,S$. Let $\theta_i=(\pi_i, p_{i1},\dots,p_{iS})$. Assume that conditional on $\theta_i$, $K_i$ is independent of $X_i$.
% Extending the model too much may result in overfitting.

The goal is to estimate $\E_G \eta(\theta)$ where $\eta(\theta_i)=(p_{i1},\dots,p_{iS})$. This provides an estimate of the population proportion of items with $X=x_s$, e.g., the proportion of unemployed individuals in the population. We assume that the tendency to respond and the tendency to be employed may be correlated. However, we assume that they are independent given the latent variables. Thus, we do not permit situations in which unemployed people are shy to respond \emph{because} they are unemployed, or employed workers are harder to sample because they are busy.

\end{example}
The following example was investigated in Greenshtein and Ritov (2022).

\begin{example} \label{ex:GR}
 
To estimate the population's proportion of (say) unemployed, a sample is designed
in $n$ ``small areas/strata". Strata are chosen to be small so that conditional on the strata, 
the ``missing at random"  assumption is approximately satisfied. 
A sample of size $\kappa$ is taken from each of the $n$ strata. Let $K_i$ be the number of responses 
in stratum $i$, $i=1,\dots,n$. The possible  observed outcomes are $(X_i, K_i)$, $i=1,\dots,n$, where $X_i$ is the number of unemployed among the $K_i$ responders, while $K_i=0,\dots,\kappa,$ is the number of responders in stratum $i$.

 Assume  the conditional 
distribution of $X_i$  given $K_i$ is binomial, $B(K_i, p_i)$, while $K_i \sim B(\kappa,  \pi_i)$.
Define  $\theta_i=(\pi_i,p_i)$,
and $\eta(\theta_i)= p_i$.  The population's unemployed proportion equals $\E_G \eta(\theta)$   (assuming strata are of equal size). The difficulty arises when $K_i=0$ for some strata. In such cases, the obvious estimator 
$n^{-1} \sum_i {X_i}/{K_i}$ cannot be applied.

 A similar setup is suitable in observational studies, where $K_i$,   the number of observed individuals in 
stratum $i$ within one unit of time, follows a  $Po(\lambda_i)$ distribution. This setup occurs, for example, when estimating
the spread of a disease, based on the proportion, $p_i$, of infected individuals in each stratum, $i$, $i=1,\dots,n$.
The estimator is based on a ``convenience sample" from small strata, of people who underwent
(possibly unrelated) blood tests.
The available data include the number of tested individuals in each stratum and the corresponding test results.
Strata are chosen to be ``small,"  so that the sampled/tested
individuals in a stratum are reasonably representative of non-sampled individuals.
Here the parameters are  $\theta_i=(\lambda_i,p_i)$ and $ \eta(\theta_i)=p_i.$
As before, the difficulty stems from strata with $K_i=0$        . 

Note, in the current case, where $K_i$ is distributed  Poisson, Greenshtein and Ritov (2022) proved that $\hat{G}^n \Rightarrow G$. Thus, the good simulation and real data analysis results, reported in Greenshtein and Ritov (2022), are less surprising.

\end{example}

In these examples, we differentiate between two types of missing observations. They can be censored or truncated. In the context of Example \ref{ex:GR}, it may be that all strata are observed, whether $K_i>0$ or not. If $K_i=0$, we know that there were no tested individuals from stratum $i$, and therefore we do not obtain any information about the prevalence of the condition in the stratum. In the truncated situation, we do not know about the existence of the stratum unless $K_i>0$ patients from the stratum get to the clinic.

 \section{On the asymptotics of the SPMLE of the observed density and its functionals}. \label{sec:theory}

\subsection{Uniqueness of  $f_{\hat{G}} (y_i)$ for the observed $y_i$.} 
\label{sec:uniqueness}

In the remainder of the paper $\hat{G} \equiv \hat{G}^n$ denotes an SPMLE based on the observed sample $t(Y_1),\dots,t(Y_n)$. When $\hat{G}$
is not unique, we refer to {\it any} SPMLE, and any corresponding estimator 
$\eta_{\hat{G}}$. It was pointed out in Ritov and Greenshtein (2022) that by considering the derivative of the likelihood of the submodel $dG_t (\th)=\bigl(1+t(\eta(\th)-\eta_{\hat G})\bigr)d\hat G(\th)$, $t\in(-\eps,\eps)$, which should be $0$ at $t=0$, we obtain the self-consistency equation:
 \eqsplit{
    \eta_{\hat G} &= \frac{1}{n}\summ i1n\E_{\hat G}\bigl(\eta(\th)\mid Y_i\bigr). 
  }
 However, self-consistency is not consistency, and in particular the connection is problematic when $\hat G$ is not unique.

To simplify notation, we do not distinguish between $Y_i$ and $t(Y_i)$. Our fundamental lemma is:

\begin{lemma} \label{lem:main}
Let $\hat{G}_0$ and $\hat{G}_1$ be two SPMLEs. Then: 
\begin{equation} \label{eq:lem1}
f_{\hat{G}_0}(Y_i)=f_{\hat{G}_1}(Y_i), \qquad i=1,\dots,n.
\end{equation}
\end{lemma}

Lemma \ref{lem:main}  is related to  Theorem 18  of Lindsay (1995). However, our proof
requires weaker assumptions and appears simpler and more directly related to the present paper.

 \begin{proof}
Recall:  
\begin{equation} \label{eqn:L}
\hat{G}=\argmax_{G} \Pi_i  f_G(Y_i) \equiv \argmax_{G} \log(\Pi_i  f_G(Y_i)) 
\equiv \argmax_G \ell(G),
\end{equation}
where
 \eqsplit{
    \ell(G)&= \log\Bigl(\summ i1n \log\int f(Y_i\mid\th)dG(\th)\Bigr)
    \\
    &=\log \Bigl(\summ i1n f_G(Y_i)\Bigr) .
  }
Suppose $\ell(G)$ is maximized by $G_0$ and $G_1$. Consider $G_\lm = (1-\lm)G_0+\lm G_1$. Then $f_{G_\lm}(Y)=(1-\lm)f_{G_0}(Y)+\lm f_{G_1}(Y)$, hence $\chi(\lm)\equiv \ell(G_\lm)$ is a concave function on $[0,1]$ with a second derivative
 \eqsplit{
    \chi''(\lm)&= -\summ i1n  \Bigl(\frac{f_{G_1}(Y_i)-f_{G_0}(Y_i)}{(1-\lm)f_{G_0}(Y_i)+\lm f_{G_1}(Y_i)}\Bigr)^2.
  }
By assumption, $\chi$ is maximized at $0$ and $1$, which can happen if, and only if, $\chi''(\lm)=0$ for $\lm\in(0,1)$ implying $f_{G_0}(Y_i)=f_{G_1}(Y_i)$, $i=1,\dots,n$.

\end{proof}

 \subsection{The truncated and censored cases.}\label{sec:trunc}

In the case where nonresponse is censored,  for items that did not respond, we do know that their corresponding value  $y$ 
satisfies
$y \notin A$. Our observations are the pairs $(Y_iI_i, I_i)$, $i=1,\dots,n$.

In the truncated setup,  we do not know about observations $Y    \notin A$.   As examples, consider the well-known models and problems: capture-recapture; how many words did Shakespeare know? and the number of unseen species. In those examples, we do not know about 
unobserved items.
The corresponding density of an observed value is
$f^A(y\mid\theta)=f(y\mid\theta)\ind(y\in A)/P_\theta(A)$. However, a latent $\theta$ that is truncated is not expressed in the sample. Thus, we have a weighted sampling from $G$, and the observed density is 
 \eqsplit[eq:trun]{
    f_G(y) &= \frac{\int f(y\mid\theta)dG(\theta)}{\int P_\theta(A)dG(\theta)} 
    \\
    &= \frac{\int f^A(y\mid\theta)P_\theta(A)dG(\theta)} {\int P_\theta(A)dG(\theta)}
    \\
    &= \int f^A(y\mid\theta)dG^A(\theta),
  }
where   $dG^A(\theta)=P_\theta(A)dG(\theta)/\int P_\vartheta(A)dG(\vartheta)$.

The extension of Lemma  \ref{lem:main}  to the censored case is immediate; it is the same mixture structure on the censored sample space with density $f_G(y)\ind(y\in A)\times P_G(A^{c})$. However, in the truncated case the density is not linear in $G$ and hence the argument of Lemma \ref{lem:main} fails. But the claim can be extended:

\begin{lemma}
   \label{lem:mainTruncated}
   Suppose $\hat G_1$ and $\hat G_2$ are SPMLEs in either the censored or the truncated cases. Then, in the censored case $f_{\hat G_1}(Y_i)I_i=f_{\hat G_2}(Y_i)I_i$, $i=1,\dots,n$,  and if there is at least one censored observation, then $P_{\hat G_1}(A^c)=P_{\hat G_2}(A^c)$. In the truncated case,  $f_{\hat G_1}(Y_i)/P_{\hat G_1}(A)=f_{\hat G_2}(Y_i)/P_{\hat G_2}(A)$.
\end{lemma}
\begin{proof}
  The censored case is simple and already argued. For the truncated case, note that the set $\{G^W\}=\{G\text{ a distribution on } \Theta,\: \int dG(\theta)/ P_\theta(A)<\infty\}$ is convex, and therefore the argument of Lemma \ref{lem:main} follows \eqref{eq:trun} with the kernel $f^T(y\mid\theta)$.
\end{proof}

\subsection{On the consistency of the $f_{\hat G}$}
\label{sec:spmleConsistent}

\begin{comment}
Some conditions are needed to ensure the consistency of the SPMLE. In particular, consistency may fail if the distribution has very long tails. For example, suppose $f_G$, for a fixed $G$, is supported on the positive integers and  for large enough $j$: $ f_G(j)=1/\bigl(j\log^2(j)\bigr)$, then $\summ jm\en f_G(j)=\O(1/\log m)$. This means that the largest order statistic  has the order $\OP({F_G^{-1}\bigl((n-1)/n\bigr)})=\OP({e^{n}})$. and thus its contribution to the average slog-likelihood is of order $ \OP({-\log(e^nn^2)/n}) =\OP({-1})$. If the family $\scg$ is rich enough to have $O(1)$ density at any $j$, the SPMLE may well be inconsistent.

To simplify the argument, we consider the following relatively strong assumption to ensure consistency:

\begin{Assumption}
  [\bf Subgeometric distribution]\label{ass:subgeom} $Y$ is discrete with support $y_1,y_2,\dots$, and $f_G$ has bounded from below hazard rate:
   \eqsplit{
    \rho\equiv\inf_m \frac{f_G(y_m)}{\summ jm\en f_G(y_j)}>0.
    }    
\end{Assumption}

\end{comment}
\def\df{\ensuremath f\hspace{-0.7em}f}

The SPMLE is consistent, and under some conditions it is even with ensured rate:

\begin{theorem}\label{thm:conv}
     Suppose $\theta\dist G$ and $Y$ is discrete with support $y_1,y_2,\dots$. Let:
    \eqsplit{
    \nu(\eps)&\eqdef\summ j1\en \ind(f_G(y_j)\ge\eps),\qquad\text{and}
    \\ 
    \xi(\eps)&\eqdef\summ j1\en f_G(y_j)\ind(f_G(y_j)\le\eps).
    }Then
     \eqsplit{
       \|f_{\hat G}-f_G\|_{TV} &=  \OP({\sqrt{\frac{\xi(L_n/n)}{\eps}+\frac{\nu(1/n)}{n\eps}}+\xi(\eps)})
       \cip 0,
      }
 where $L_n$ is some slowly growing to infinite sequence.     

  If further, $Y$ has sub-geometric tails: 
  \eqsplit[ass:subgeom]{
    \rho\equiv\inf_m \frac{f_G(y_m)}{\summ jm\en f_G(y_j)}>0.
    }    
%\eqref{ass:subgeom} holds. 
  Then $\|f_{\hat G}-f_G\|_{TV}=\OP({n^{1/3}})$.
\end{theorem}
\begin{proof}

  Let $\df_n$ be the empirical density function of $Y$. Since $\df_n$ maximizes the likelihood among all distributions and $f_{\hat G}$ among all distributions in the family, we obtain
   \eqsplit[sand]{
   \summ j1\en \df_n(y_j)\log f_G(y_j)
   \le
   \summ j1\en \df_n(y_j)\log f_{\hat G}(y_j)
   \le
    \summ j1\en \df_n(y_j)\log \df_n(y_j).
    }
  Let $L_n$ be a sequence which is slowly growing to infinity, and which may be different from line to line in the proof. The difference between the RHS and the LHS of \eqref{sand} can be bounded  
  \eqsplit{
   0&\ge \summ j1\en \df_n(y_j)\log\Bigl(\frac{f_G(y_j)}{\df_n(y_j)}\Bigr)
   \\
   &= \sum_{f_G(y_j)> L_n/n} \df_n(y_j)\log\Bigl(\frac{f_G(y_j)}{\df_n(y_j)}\Bigr)
   + \sum_{f_G(y_j)<L_n/n} \df_n(y_j)\log\Bigl(\frac{f_G(y_j)}{\df_n(y_j)}\Bigr)
   \\
   &\ge -\sum_{f_G(y_j)> L_n/n} \bigl(f_G(y_j)-\df_n(y_j)\bigr)\ind(\df_n(y_j)>0)
   \\
   &\hspace{3em}-\sum_{f_G(y_j)> L_n/n}\frac{\bigl(\df_n(y_j)-f_G(y_j)\bigr)^2}{\df_n(y_j)}\ind(\df_n(y_j)>0)
    \\
    &\hspace{3em}-\xi({L_n}/{n})\log(n\xi^{-1}(1/L_nn)),\qquad \text{w.p}\to 1,
   \\
   &\ge  -\xi({L_n}/{n}) -
   -\frac{\nu(1/n)}{n},\qquad \text{w.p}\to 1,
       }
  since, by Assumption \eqref{ass:subgeom} the total mass with $f_G$ smaller than $L_n/n$ is $\xi(L_n/n)$,  with probability converging to 1, there are no $Y_i$s  with $f_G(Y_i)<\xi^{-1}(1/L_nn)$, and $\df_n/f_G=\OP({\sqrt{f_G/n}})$ on its support. Thus, with probability converging to 1:
   \eqsplit{
     &\hspace{-3em}\xi(L_n/n)+\frac{\nu(1/n)}{n}
     \\
     &\ge -\summ j1\en \df_n(y_j)\log\Bigl(\frac{f_{\hat G}(y_j)}{\df_n(y_j)}\Bigr) 
     \\
     &= -\summ j1\en \df_n(y_j)\log\Bigl(1+\al \frac{f_{\hat G}(y_j)-\df_n(y_j)}{\df_n(y_j)}\Bigr)  
     \\
     &= \summ j1\en f_{\hat G}(y_j)\ind(\df_n(y_j)=0)
     \\
     &\hspace{3em}+ \summ j1\en \frac{\bigl(f_{\hat G}(y_j)-\df_n(y_j)\bigr)^2}{\df_n(y_j)} \Bigl(\frac{\df_n(y_j)}{\al^* f_{\hat G}(y_j)+(1-\al^*)\df_n(y_j)}\Bigr)^2
     \\
     &\ge \OP({\eps})\sum_{j\in\scm(\eps)}\bigl(f_{\hat G}(y_j)-\df_n(y_j)\bigr)^2
    }
where $\scm(\eps)=\{j: \df_n(y_j)>\eps\}$. Note that by Assumption \eqref{ass:subgeom}: $|\scm(\eps)|=\OP({\nu(\eps)})$   and $\sum_{j\not\in\scm}f_G(y_j)=\OP({\xi(\eps)})$. Thus,
 \eqsplit{
    \summ j1\en &|f_{\hat G}(y_j)-f_{G}(y_j)|
    \\
    &\le 2\sum_{j\in\scm(\eps)} |f_{\hat G}(y_j)-f_{G}(y_j)| +2\sum_{j\not\in\scm(\eps)}f_{G}(y_j)
    \\
    &\le \sum_{j\in\scm(\eps)} |f_{\hat G}(y_j)-\df_n(y_j)| 
    + \sum_{j\in\scm(\eps)} |f_{G}(y_j)-\df_n(y_j)| +2\xi(\eps)
    \\
    &\le \bigl(1+\op(1)\bigr)\Bigl(\sum_{j\in\scm}\bigl(f_{\hat G}(y_j)-\df_n(y_j)\bigr)^2\nu(\eps)\Bigr)^{1/2}  +2\xi(\eps)
    \\
    &\hspace{3em}
    +\bigl(1+\op(1)\bigr)\Bigl(\sum_{j\in\scm}\bigl(f_{ G}(y_j)-\df_n(y_j)\bigr)^2\nu(\eps)\bigr)\Bigr)^{1/2} +2\xi(\eps)
    \\
    &\le \bigl(1+\op(1)\bigr)\Bigl(\sum_{j\in\scm}\bigl(f_{\hat G}(y_j)-\df_n(y_j)\bigr)^2\nu(\eps)\bigr)\Bigr)^{1/2} +\OP({\frac{\nu(\eps)}{n}})  +2\xi(\eps)
    \\
    &=\OP({\sqrt{\frac{\xi(L_n/n)}{\eps}+\frac{\nu(1/n)}{n\eps}}}) +\OP({\frac{\nu(\eps)}{n}}) +  \OP(\xi(\eps)). 
  }
  Since, $\nu(\eps)/n\le (n\eps)^{-1}\le(n\eps)^{-1/2}$, we obtain:
   \eqsplit{
    \|f_{\hat G}-f_G\|_{TV} &= \OP({\sqrt{\frac{\xi(L_n/n)}{\eps}+\frac{\nu(1/n)}{n\eps}}+\xi(\eps)}).
    }
   Since, for any $\eps>0$, $\nu(1/n)/n\to 0$, and $\xi(t)\to 0$, we obtain part 1. Part 2 follows with $\nu(\eps)=\O({|\log\eps}|)$ and $\xi(\eps)=\O({\eps})$ by Assumption \eqref{ass:subgeom}.

\end{proof}

\begin{example}
  Theorem \ref{thm:conv} does not ensure uniform convergence in triangular arrays. Consider that at stage $n$, we observe $n$ Poisson variables with mixing distribution uniform on the interval $(0,m)$, where $m=m_n> n^2$. Then, $Y$ itself has a distribution that is approximately uniform on the integers $1,2,\dots,m$. Thus the sample is sparse. Since a point mass at $k$ is the best explanation of the observation at $k$, and explains poorly any other observations, the SPMLE will be close to the empirical distribution of the $Y$. Thus, the SPMLE does not converge to the true distribution in total variation. However, scaled properly, it converges to the uniform distribution, and thus $\scl_{\hat G}(Y/m)\cweak \scl_{G}(./m)$ where $\scl_G$ is the cdf of $Y$ under $G$.
  
  Consider now an example where we have that the SPMLE converges weakly to the wrong distribution. Suppose, as before, that $G$ is the uniform distribution on $(0,m)$ and $Y$ is discrete. Given $\th$, $Y=\lfloor\th\rfloor$ with probability $\th/m$, where $\lfloor x\rfloor$ is the integer part of $x$. With the remaining probability of $1-\th/m$ it is distributed uniformly on the integers $\lfloor \th\rfloor+1,\dots,m$. Thus, the $Y$ has the so called star-shape density $f_G(y)=yg(y)+G(y)/m$, $y=0,1,\dots,m$, where $g(y)=G(y)-G(y-1)$. The first term of $f_G$ dominates since $yg(y)=\O(1)$ while $G(y)/m=\O({m^{-1}})$. Asymptotically, the SPMLE is  
   \eqsplit{
    \argmax_g\summ i1n \log(Y_ig(Y_i)+G(y)/m)&\approx \argmax_g\summ i1n \log (Y_i g(Y_i))
    \\
    & =\argmax_g\Bigl(\summ i1n \log(Y_i)+\summ i1n \log(g(Y_i))\Bigr) 
    \\
    & =\argmax_g\Bigl(\summ i1n \log(g(Y_i))\Bigr) ,
    }
     i.e., $\hat G$ is close to the distribution of $Y$ which is approximately $yG(y)/m$, i.e., there is an extra $y$ factor which will not disappear after rescaling.  This example is built upon the inconsistency of the star-shaped distribution on the unit interval, which we learned from Willem van Zwet. 
  
  On the positive side, suppose that given $\th$, $Y$ is Poisson with mean $\th$. Assume further that (i) The median of $G$ is greater than $\al>0$, and (ii) $\int \th^2 dG(\th)<\beta$. For any large $m$, the first condition established a lower bound on $\min_{j\le m}f_G(j)$, while the second condition establish a bound on $\summ jm\en f_G(j)\le 2^{-m/2}+4\beta/m^2$. Thus, $\xi(\eps)\to 0$ uniformly in this class, and the convergence in total variation is uniform over this class.
\end{example}

\subsection{The identifiable functionals}
\label{sec:ident}

\begin{Assumption}\label{ass:hfun}
Consider a sequence $\bby=y_1,y_2,\dots$ that is dense in the support of $f_G$. Let $f_G(\bby_{1:m})=\bigl(f_G(y_1),\dots, f_G(y_m)\bigr)$ and $f_G(\bby)=\bigl(f_G(y_1), f_G(y_2),\dots\bigr)$, where $f_G(\cdot)$ is the observed density. Suppose:
\begin{enumerate}[(i)]
  \item\label{ass:hfun1} There is a function $h:\scy^{\en}\to\R$ such that for all $G$, $\eta_G=h\bigl(f_G(\bby)\bigr)$.
  \item\label{ass:hfun2} There are continuous  functions, $h_m: \scy^m\to \R$, such that $h_m\bigl(f_G(\bby_{1:m})\bigr)\to h\bigl(f_G(\bby)\bigr)$.
  \item\label{ass:hfun3} For any sequence $\hat G_n$ and any $m=1,2,\dots$: $\bigl(f_{\hat G_n}(y_1),\dots,f_{\hat G_n}(y_m)\bigr) \cip \bigl(f_{G}(y_1),\dots,f_{G}(y_m)\bigr)$. 
\end{enumerate}
\end{Assumption}

Assumption \ref{ass:hfun} is intended to be applied to discrete distributions like the Poisson, but was written in a way that would hold more generally, for example, in the Gaussian shift model. Assumption \ref{ass:hfun}\eqref{ass:hfun1} is an application of the likelihood principle to an identified parameter. Assumptions \ref{ass:hfun}\eqref{ass:hfun2}\&\eqref{ass:hfun3} ensures that $h$ is not a tail property. 

\begin{theorem}\label{th:funct}
Suppose $Y$ is discrete and Assumption \ref{ass:hfun} holds. Then for any sequence $\hat G_n$ of SPMLEs, $\eta_{\hat G_n}\cip \eta_G$.
\end{theorem}
\begin{proof}
  For any $\eps>0$, fix $m$ such that $\Bigl|h_m\bigl(f_G(\bby_{1:m})\bigr)-h\bigl(f_G(\bby)\bigr)\Bigr|<\eps$. For $n$ large enough, $\{y_1,\dots,y_m\}\subseteq \{Y_1,\dots, Y_n\}$, and hence $f_{\hat G}(\bby_{1:m})$ does not depend on the version of the SPMLE, and by Assumption \ref{ass:hfun}\eqref{ass:hfun3} $h\bigl(f_{\hat G}(\bby_{1:m})\bigr)$ converges in probability to $h\bigl(f_{\hat G}(\bby_{1:m})\bigr)$.
\end{proof}

%\subsubsection{A few more examples}
\label{sec:moreExamples}

\begin{example}
  Suppose $Y\dist B(K,p)$. The functions that can be estimated directly are $f_G(0),\dots,f_G(K)$. However, $f_G(j)=\binom{K}{j}\E_G\bigl( p^j (1-p)^{K-j}\bigr) = \binom{K}{j}\summ mjK \binom{K-j}{m-j}(-1)^{m-j}\E_Gp^m.$ Thus, the functions that can be estimated are of the form $h\bigl(\mu_1(G),\dots,\mu_K(G)\bigr)$, where $\mu_j(G)$ is the $j$th moment of  $G$ and $h(\cdot)$ is a continuous function. 
  
  The SPMLE is not unique. This could be expected since the distribution of $Y$ belongs to a subset of the $K$ dimensional nonparametric family of all distributions supported on $0,\dots, K$. Suppose $\hat G$ is a SPMLE. Any
   \eqsplit{
    G\in \scg_n\eqdef \bigl\{G: \E_G(p^y(1-p)^{K-y})=\E_{\hat G}(p^y(1-p)^{K-y}),\;y\in\{Y_1,\dots,Y_n\}\bigr\}
    }
  is an SPMLE.  Note that $\scg_n\not\weakly G$.

  If $\eta(p)=p$, $\eta_G$ is easily estimated by $\bar Y_n$. Asymptotically, $\eta_{\hat G}=\bar Y_n$ for any SPMLE $\hat G$. However, this is not necessarily the case for finite samples. For examples, suppose $K=3$, and the sample obtained has $Y_i\in\{2,3\}$, $i=1,\dots,n$. Then, for any two SPMLEs, $\mu_2(\hat G_1)=\mu_2(\hat G_2)$ and  $\mu_3(\hat G_1)=\mu_3(\hat G_2)$, but it may be that  $\mu_1(\hat G_1)\ne\mu_1(\hat G_2)$, and hence $\eta_{\hat G_1}\ne\eta_{\hat G_2}$.
  
  Suppose the sample is truncated and only positive values of $Y$ are observed and counted. In this case, what can be estimated is only 
   \eqsplit{
    \frac{f_G(y)}{1-f_G(0)} &= \frac{\binom{K}{j}\summ mjK \binom{K-j}{m-j}(-1)^{m-j}\mu_G(m)}{\summ m1K (-1)^m \binom{K}{m} \mu_G(m)}.
    }
Thus, only functions of the form   $h\bigl(\mu_2(G)/\mu_1(G),\dots,\mu_K(G)/\mu_1(G)\bigr)$  for continuous $h(\cdot)$ can be estimated. In particular, neither $\E_G(p)$ nor $\var_p(p)$ can be estimated, but the coefficient of determination of $G$ can be.
\end{example}

\begin{example}\label{ex:pois}
  Suppose $Y\dist Po(\lm)$. Let $\scy_n=\{y_1,\dots,y_k\}$---the observed unique values, i.e., the minimal set such that $Y_i\in\scy_n$, $i=1,\dots,n$. If $\hat G$ is SPMLE, any 
   \eqsplit{
    G\in\scg_n\eqdef\{G:\; \int \lm^y e^{-\lm}dG(\lm)= \int \lm^y e^{-\lm}d\hat G(\lm), y\in\scy_n \}
    }
   is an SPMLE. Unlike the binomial case, since $\scy_n$ is asymptotically increasing to $\{0,1,2,\dots\}$, $\scg_n$ is shrinking towards the truth.
   
   Since $\int \lm^{j}e^{-\lm} dG(\lm)=j! f_G(j)$, this expectations can be estimated, and therefor $\eta_G$ for any function $\eta(\cdot)$ such that $\eta(\lm)e^{\lm}$ can be approximated by a power series. 
   \end{example}
   
   \begin{example}
   So far we considered mostly linear functions of the observed density. Consider the Poisson distribution of Example \ref{ex:pois}, except that only strictly positive values are observed, or, in other words, we consider a sample truncated at $Y>0$. In yet other words, we observe a sample from  
    \eqsplit{
        f_G(y) &= \frac{\int \lm^ye^{-\lm}dG(\lm) } {y!\bigl(1-\int e^{-\lm}dG(\lm)\bigr)}. 
     }
     
     The functions that can be directly estimated are $h_j(G)\equiv\int \lm^j e^{-\lm}dG(\lm)/\bigl(1-\int e^{-\lm}dG(\lm)\bigr)$, $j=1,2,\dots$.
   Suppose it is known that $G(\eps)=0$, for some small $\eps>0$. Suppose $\bigl\|\bigl(\summ j1M \al_{Mj}\lm^j-\lm^{-1}\bigr)e^{-\lm}\ind(\lm>\eps)\Bigr\|_\en\to 0$. Then $\summ j1M \al_{Mj}\hat h_j$ is an estimator of 
   \eqsplit{
        \frac{\int e^{-\lm}dG(\lm)} {1-\int e^{-\lm}dG(\lm)} &= \frac{f_G(0)}{1-f_G(0)} ,
   }
   and $f_G(0)$ can be estimated by
    \eqsplit{
        \hat f_G(0)&= \frac{\summ j1M \al_{Mj}\hat h_j} {1+\summ j1M \al_{Mj}\hat h_j} .
     }

  \end{example}

\begin{example}\label{ex:poest}
  Suppose at each unit we sample a Poisson number of subjects, and check the employment status of the $K$ sampled subjects. Let $X$ be the number of employed. We want to estimate the marginal probability of being employed. Thus, $f_G(x,k)=\binom{k}{x}\frac{1}{k!}\dint p^{x}(1-p)^{k-x}\lm^{k}e^{-\lm}dG(p,\lm)$, for $0\le x\le k$ and $k\ge 0$, and $\eta_G=\dint pdG(p,\lm)$. We should prove that $\eta_{\hat G}$ is a consistent estimator. The difficulty in this case is that we have no unbiased information on $p_i$ in units with $K_i=0$. Thus, we have only biased information in units with $\lm_i$ small, and since $\lm$ and $p$ may be correlated, the information in the sampled subjects is biased.
  
  Let
   \eqsplit{
    \ti f_G(k) &= \summ x0k \frac{x}{k}f_G(x,k)
    \\
    &= \int p \frac{\lm^k}{k!}e^{-\lm}dG(p,\lm),\qquad k=1,2,\dots.  
    }
It follows that $\summ k1\en \ti f_G(k)=\int p(e^{\lm}-1)e^{-\lm}dG(p,\lm)=\eta_G - \int p e^{-\lm}dG(p,\lm)$. To argue that $\eta_G$ can be estimated by $\eta_{\hat G}$ we can argue that $\int p e^{-\lm}dG(p,\lm)$ is in the closure of the  span of $\{\ti f_{G}(k):, k=1,2,\dots\}$. Note that,
 \eqsplit{
   \summ k1{m+1}\al_k \ti f_G(k) &= \int p \lm e^{-\lm} \summ j0{m} \al_{j+1}\frac{k^j}{(j+1)!} dG(p,\lm)
   \\
   &=  \int p \lm e^{-\lm} Q_m(k) dG(p,\lm)
  }
for some polynomial $Q_m$. Thus we need to approximate in an appropriate sense $\lm^{-1}e^{\lm}$ by a polynomial. 

Let $u_\gamma(\lm)=\lm^{-1}e^{\lm-\gamma\lm-\gamma\lm^2}$. Then, $u_\gamma$ is bounded, square integrable, and $$\lim_{\gamma\to 0}\int p\lm e^{-\lm}u_\gamma(\lm) dG(p,\lm)=\eta_G.$$ For every $\gamma>0$, $u_\gamma(\cdot)$ can be approximated by its projection  $Q_{\gamma m}$ on the span of the Laguerre's polynomials of degree not larger than $m$. The Laguerre's polynomials are a family of orthogonal polynomials with respect to the weight function $e^{-\lm}$ on the positive reals,  such that 
 \eqsplit{
    \lim_{m\to\en} \int \bigl(u_\gamma(\lm)-Q_{\gamma m}(\lm)\bigr)^2 e^{-\lm}d\lm=0.
  }
Thus, there is a sequence $\gamma_m\dec 0$ such that  
 \eqsplit{
    \lim_{m\to\en}\int p\lm e^{-\lm}Q_{\gamma_m,m}(\lm)dG(p,\lm) &= \eta_G.
  }
Which can be mapped to a sequence $\al_{mk}$ such that
 \eqsplit{
    h_m(f_G(1),\dots, f_G(m)) &\equiv \summ k1m \al_{mk}\ti f_G(k)\to \eta_G.
  }
  
By Theorem \ref{th:funct}, we do not need to exhibit the approximation polynomials $Q_{\gamma_m,m}$ or the sequence $h_m$ explicitly to argue that $\eta_{\hat G}\cip \eta_G$.

\end{example}

\begin{example}
  We consider the same setup as in Example \ref{ex:poest}, except that we have a truncated sample and do not observe when it was supposed to be $K=0$. Thus, 
    \eqsplit{
    \ti f_G(k) &= \summ x0k \frac{x}{k}f_G(x,k)
    \\
    &= \frac{\int p \frac{\lm^k}{k!}e^{-\lm}dG(p,\lm)} {1-\int e^{-\lm}dG(p,\lm) }, \qquad k=1,2,\dots.  
    }
 The same linear combination as in Example \ref{ex:poest}  yields   
 \eqsplit[etatrunc]{
    \lim_{m\to\en}\summ k1m \al_{mk}\ti f_G(k) &= \frac{\eta_G} {1-\int e^{-\lm}dP(p,\lm)}.
  }
However, similar  polynomials, $Q^*_{\al_m,m}(\lm)$ and similar procedures applied to the marginal density of $K$ to approximate $\lm^{-1}$ (without the exponent), satisfy
 \eqsplit{
    \lim_{m\to\en}\summ k1m \al^*_{mk} \summ x0k f_G(x,k) &= \frac {\int e^{-\lm}dP(p,\lm)} {1-\int e^{-\lm}dP(p,\lm)}.
  }
Thus, the denominator in \eqref{etatrunc} can be estimated using the observable, and, hence, $\eta_G$ can be approximated as the limit of the following functions of the densities:
 \eqsplit{
    h_m(f_G(1),\dots, f_G(m)) &\equiv \Bigl(\summ k1m \al_{mk}\summ x0k \frac xk f_G(x,k)\Bigr) \Bigl(1+\summ k1m \al^*_{mk} \summ x0k f_G(x,k)\Bigr).
  }
  
\end{example}

\subsection{Confidence intervals and partially identifiable models}
\label{sec:CI}

Consider the situation of Example \ref{ex:GR}. The probabilities of the observed $Y_i$ are proportional to $\pi^{k}(1-\pi)^{\kappa-k}p^j(1-p)^{k-j}$, $k=0,\dots,\kappa$ and $j=0,\dots,k$. Thus, the SPMLE is not unique, and the only functionals that can be estimated are functions of  $\E_G(\pi^k p^{j})$, $k=1,\dots,\kappa$ and $j=0,\dots,k$. In particular $\eta_G\equiv\E_G(p)$ is not identifiable. However, not all possible values of $\eta_G$ are consistent with the data. This becomes a partially identified model, see Chesher and Rosen(2015) and Ho and Rosen(2017).

One can ask the following two questions:
\begin{enumerate}
  \item What are $\inf\{\eta_{\hat G}:\; \hat G \text{ is SPMLE}\}$ and $\sup\{\eta_{\hat G}:\; \hat G \text{ is SPMLE}\}$?
  \item What are the minimal and maximal values of $\eta_G$ that the data cannot reject?
\end{enumerate}

By Lemma \ref{lem:main}, given a single realization of the SPMLE, all other SPMLEs are defined by the linear constraints and, thus, the first question can be answered using standard linear programming:
 \eqsplit{
    &\text{Find } \min_G \pm \int \eta(\theta) dG(\theta)
    \intertext{s.t.}
    &\int f(y_i\mid \theta)dG(\theta)=f_{\hat G}(y_i),\quad i=1,\dots,n,
    \\
    &dG(\theta)\geq0.
  }
  
Alternatively, it can be as easily answered by adding a linear constraint to the SPMLE equation: 
 \eqsplit{
    &\text{Let }\ell(G)=\sum_{i=1}^{n}\log f_G(Y_i)
    \text{ and }\hat G_\eta=\argmax_G \ell(G)\text{ s.t. } \eta_G=\eta.
    \\
    &\text{Find: } \eta_G^{+}=\sup \{\eta: \ell(\hat G_\eta)=\max_G\ell(G)\} \;\&\;
     \eta_G^{-}=\inf \{\eta: \ell(\hat G_\eta)=\max_G\ell(G)\}.
  }
 A modification of this program answers the second question, Following Wilks's Theorem, we consider
  \eqsplit{
    \eta_G^{U}&=\sup \{\eta: \ell(\hat G_\eta)>\max_G\ell(G)-\frac12\chi^2_\al\} \
    \\
     \eta_G^{L}&=\inf \{\eta: \ell(\hat G_\eta)>\max_G\ell(G)-\frac12\chi^2_\al\},
   }
 where $\chi^2_\al$ is the $1-\al$ upper quantile to the $\chi^2$ distribution with one degree of freedom.

The simple algorithm for calculating the MLE in a mixing model is the EM algorithm. The support of $G$ is discretized, and we compute the MLE supported on that grid. The E step is the usual, where each observation is spread over the support of $G$ according to the posterior distribution, assuming $G$. Let $M$ be the pseudo-sample over the support. The M step for the MLE is 
 \eqsplit{
    \argmax_g \sum_m M_m\log g(\theta_m)\quad\text{ s.t. }\sum_m g (\theta_m)=1 .
  }
We modify it to   
 \eqsplit{
    \argmax_g \sum_m M_m\log g (\theta_m)\quad\text{ s.t. }\sum_m g (\theta_m)=1 \;\&\; \sum_m \eta(\theta_m)g (\theta_m)=\eta .
  }
  The solution of the M step is
   \eqsplit{
    g(\theta_m)&=\frac{M_m}{n+\gamma\eta(\theta_m)}.
    }
 We solve in parallel the MLE and the constraint problem where we use the above solution, where $\gamma$ is adapted to ensure a difference of $\chi^2_\al/2$ between the two likelihoods. See Algorithm \ref{alg:em} for details.

\begin{algorithm}[ht]
\label{alg:em}
\caption{EM with partially identified mean}
\KwIn{Observations: $Y_1,\dots,Y_n$, Grid: $\theta_1,\dots,\theta_m$, Level, $\al_0$}

\For{$\al=0,-\al_0,\al_0$}{$g(\al,\cdot)\gets\ind_m/m$
    \\ $\gamma(\cdot)\gets 0,-\gamma_0,\gamma_0$}
    
\While{Change $>$  TINY}
{
\For{$\al=0,-\al_0,\al_0$}
{
\For {a=1:m}
    {$M(a)\gets\summ i1n \frac{f(Y_i\mid\theta_a) g(\al,a)}{\summ b1m f(y_i\mid\theta_b)g(\al,b)}$
    \\
    $g(\al,a)\gets \frac{M(a)}{n+\gamma(\al) \eta(a)}$}
    $\ell(\al)\gets \summ i1n \log\bigl(\summ a1m g(\al,a)f(Y_i\mid\theta_a)\bigr)$    
}
\For{$\al=-\al_0,\al_0$}{$\gamma(\al)\gets \gamma(\al)e^{\ell(0)-\ell(\al)-\chi^2_{1,\al_0}}$}
}   
\KwOut{g}.
\end{algorithm}

Consider now a truncated situation. In this case, the algorithm should be modified. The Markov kernel taking from $\theta$ to $Y$ is $f^{A}(y\mid\theta)$ as is in \eqref{eq:trun}. The latent values of $\theta$ in the sample are not taken from $G$ but from $G^A$. Thus, using the right kernel $f^A$, the EM algorithm yields $G^A$, and 
 \eqsplit{
    \eta_G = \frac{\int \frac{\eta(\theta)}{P_\theta(A)}dG^A(\theta)} {\int \frac{1}{P_\theta(A)}dG^A(\theta)}.
  }

\begin{example}\label{ex:sim}
  
We present now a simulation study in which we revisit the censored case of Example \ref{ex:GR} using the approach of Section \ref{sec:CI}. 
The probabilities of the different observations are functions of $\E(p^x(1-p)^{k-x})\pi^{k}(1-\pi)^{\kappa-k}$, $0\le x\le k\le \kappa$. Thus the functionals that can be identified are only functions of $p^{x}\pi^{k}$, for the same range of $x$ and $k$. In particular, $\E(p)$ is not identified. However, the simplex of distribution functions satisfying the above moment conditions is a strict subset of all distributions, and the range of possible $\E(p)$ is limited. Our approach enables a simple modification of the EM algorithm to identify the extreme value of $\eta_G$ in this simplex without the need to identify it implicitly.

We consider four sets of parameters $(\al,\beta)$, see Table \ref{tab:sim}. For each set, $M=1000$ simulations were done, where at simulation $s$, $\pi\dist Be(\al,\beta)$, and   $p\mid\pi \dist Be(6\xi_s, 6(1-\xi_s)$, where $\xi_s=(1-s/M)\pi+(s/M)(1-\pi)$. The sample size was 1000 in all cases. We implemented an EM algorithm described in Algorithm \ref{alg:em} in Section \ref{sec:CI} with a $100\times100$ grid. The percentage of observations with $K=0$, and thus missing response, is given in Table \ref{tab:sim}. This probability is the width of the naive interval, obtained by assigning either 0 or 1  to all subjects with $K=0$ and thus with missing $X$. Whenever the percentage of missing data was not negligible, the MLE-based confidence band was much smaller. A detailed presentation of the simulations is given in Figure \ref{fig:sim}.
In Figure \ref{fig:profile} we give the profile log-likelihood of $\E(p)$.

\begin{table}[th]
  \centering
  \begin{tabular}{|r|rr|r|r|}
    \hline
    % after \\: \hline or \cline{col1-col2} \cline{col3-col4} \dots
      $\kappa$ & $\al$ & $\beta$ & $P(K=0)$ & Mean CI width \\
    \hline 
  2 &  1 &  4 &  0.667 &   0.41 \\
  5 &  1 &  4 &  0.444 &   0.24 \\
  5 &  4 &  1 &  0.008 &   0.02 \\
 10 &  3 & 12 &  0.179 &   0.07  \\
    \hline
  \end{tabular}
  \caption{Mean confidence width (over 1000 simulations) and mean naive confidence bound}\label{tab:sim}
\end{table}

\begin{figure}[th]
  \centering
  \includegraphics[width=0.8\textwidth]{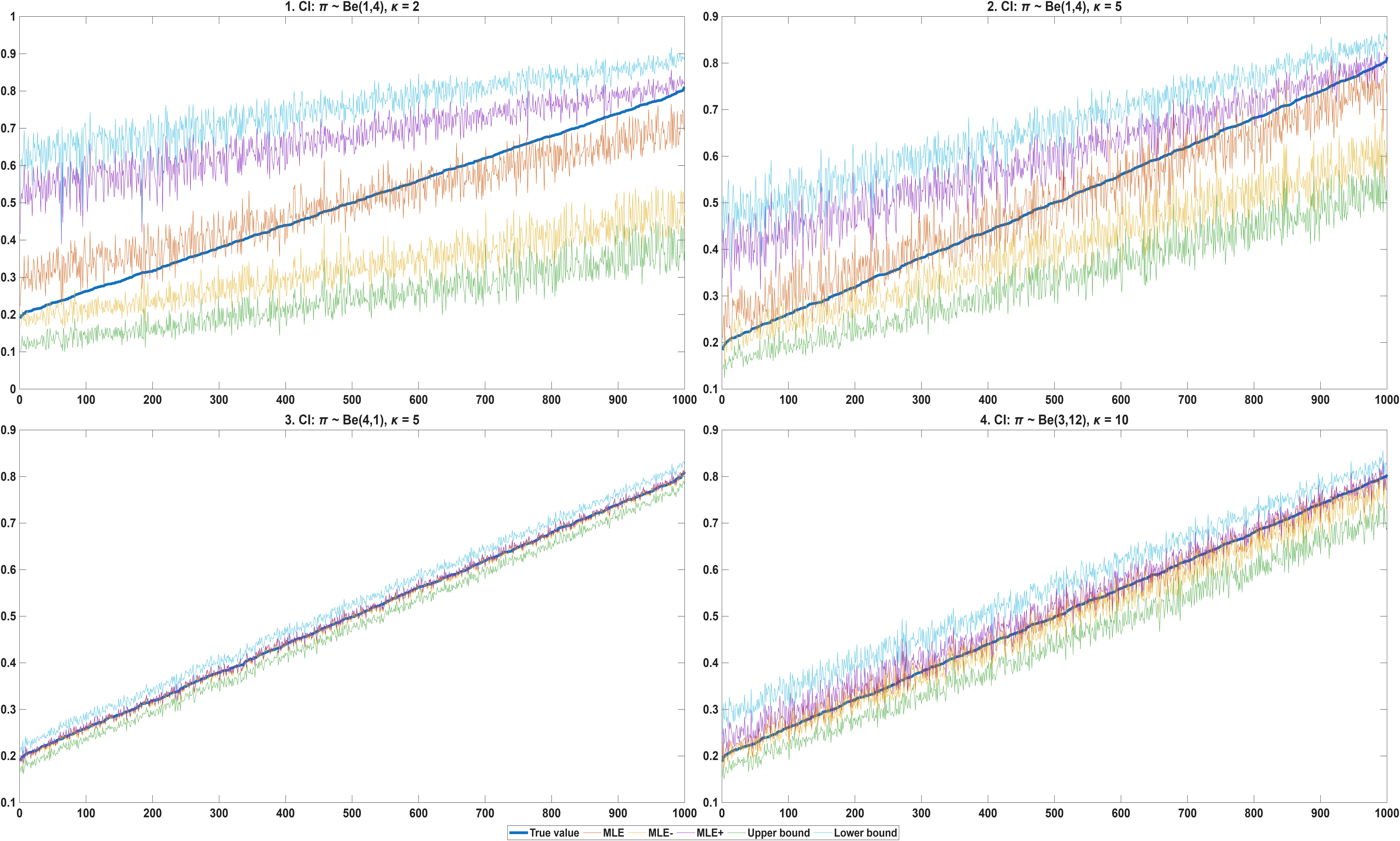}
  \caption{Four sets of simulations for missing data. The wide line is the true mean of the latent $p$. The thin lines are, from top to bottom, the upper confidence, the upper range of the MLE, the MLE---the output of the EM algorithm, the lower range of the MLE, and the lower confidence bound. }\label{fig:sim}
\end{figure}

\begin{figure}[th]
  \centering
  \includegraphics[width=0.8\textwidth]{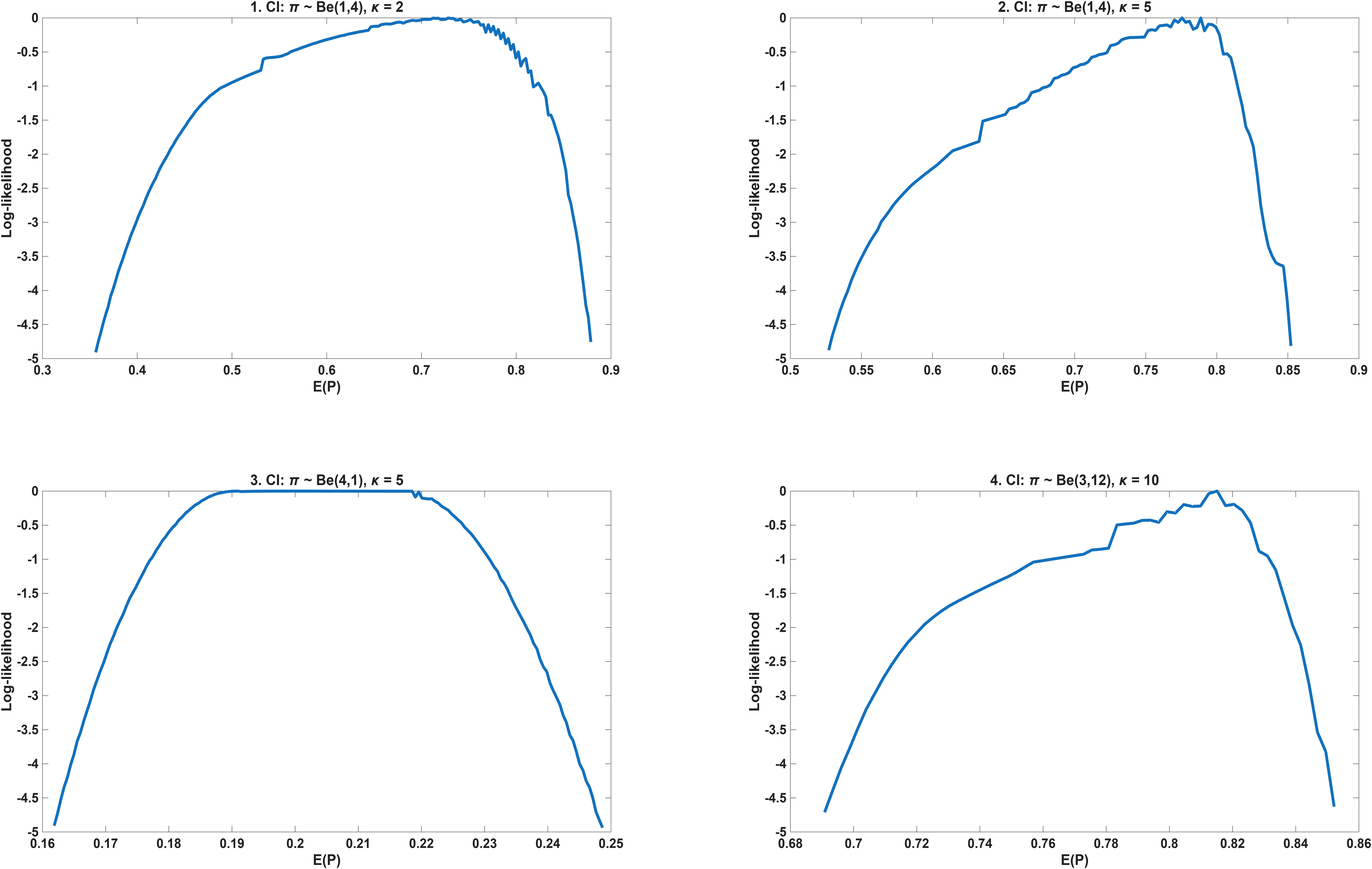}
  \caption{Profile likelihood for the 4 simulation conditions.}\label{fig:profile}
\end{figure}
\end{example}

\begin{example}
\label{ex:data}
\begin{comment}
Years 2011--2015.
Cut: of no more than 50 complaints.
3060 companies out of 3588 (0.85) in the datasets.
Responsible for 25629 out of 503670 (0.05).
Correlation between X & K: 0.28.
The proportions with K==1 and K==2 are 0.28 and 0.13. Estimated missing: 0.2 and 0.8.
The naive estimator is 0.054 (0.048,0.061).
MLE, 0.068 (0.047,0.111)
\end{comment}

  We apply our method to the record of consumer complaints against financial institutions in 2011--2015, which can be found at \url{https://www.consumerfinance.gov/data-research/consumer-complaints}. We considered only the small companies with fewer than 50 complaints during the five years. There were in total 503,670 complaints against 3588 companies; 3060 (85\%) of them were small, but these small companies are responsible for only 5\% of the complaints. Against 28\% of the small companies, there was only one complaint, and against 13\% of the institutions, there were two. Thus, there is a substantial number of potential companies for which no complaints were given. In fact, we estimate that 76\% of the potential companies are with de facto no complaints. Applying the missing species formula of Good and Toulmin (1956) gives that against 20\% of the companies would have the first complaints in the following five-year period. 
  
  The question we asked is what the average response rate of a company is with monetary relief. Thus, the units are the companies and not the customers or the complaints. It is reasonable to assume that the number of complaints against a company is correlated with the way it treats its customers (there is a 0.28 correlation between the number of complaints and the number of those resolved with monetary relief. The data was truncated by us at 50 complaints, but there are only 2\% more companies in the range of 51--60. 
  
  The data is truncated, and we observe only companies with at least one complaint. We model the data as if the number $K$ of complaints against a company is $Po(\lm)$, and given $K$, $X(K,p)$ of them end with monetary relief. We assume that $(\lm,p)$ are unobserved and follow an unknown distribution $G$.  We apply the modified EM algorithm for truncated data as described above. Since both $p$ and $\lm$ are relatively small, we consider $G$ to be supported on a $500\times 500$ equally spaced grid on the logarithmic scale.
   
   We estimated the probability of monetary relief by  $0.068$ with an uncertainty interval of $ (0.047,0.111)$. The naive estimator, the mean of the observed $X_i/K_i$, equals $0.054$, which is in the range, but it gives a false certainty of confidence interval of   $(0.048,0.061)$, which is actually all below the MLE.
\end{example}

\section{ Future work on weighted averages and  incorporating covariates.} \label{sec:weighted}

It may be of interest to estimate
weighted averages. For example, when the strata are not of equal size, and the population average is not a simple average of the strata's averages. Similarly, this arises when the sampling of strata or of the individuals is not
performed with
equal probabilities.
Suppose the relative weights are $\gamma_1, \dots \gamma_n$, and it is desired to estimate $\sum \gamma_i \eta(\theta_i)$, e.g., population average rather than simple average of strata's averages.

We now sketch a plausible approach. 
The idea is to introduce additional random variables and  corresponding parameters that will incorporate the weights
$\gamma_i$.  We demonstrate this through Example 3.

Suppose the size of stratum $i$ is $N_i$. Following a super-population model, we treat $N_i$ as
a realization of a  Poisson random variable with parameter $\lambda _i$. 
Hence, we observe $Y_i=(X_i,K_i, N_i)$.
Now, conditional  on $\theta_i=(\pi_i, p_i,\lambda_i)$,
$(X_i,K_i)$ follow the same  distribution as before,  while $N_i \sim Po(\lambda_i)$ is independent of 
$(X_i,K_i)$. We define  $\eta(\theta_i)= p_i \lambda_i$. 

The above approach might be helpful more generally to handle empirical Bayes problems involving covariates. 
By treating the covariates as realizations
from a super-population, distributed according to some parametric family, we create 
``apparent exchangeability" among the observations $Y_i$. This, in turn, strengthens the appeal of the NPEB approach.

\newpage

{\bf \Large References:}

\begin{list}{}{\setlength{\itemindent}{-1em}\setlength{\itemsep}{0.5em}}

\item Chesher, A., and Rosen, A. (2015). Characterizations of identified sets delivered by structural econometric models (No. CWP63/15). Cemmap working paper.

\item
Chen, J. (2017).  ``Consistency of the MLE under Mixture Models.'' Statist. Sci. 32 (1) 47 - 63.
\item
Eckles, D., Ignatiadis, N., Wager, S., and Wu, H. (2025). Noise-induced randomization in regression discontinuity designs. Biometrika  112(2).
\item
Dobronyi, C.,  Gu, J., and K. I. Kim (2021). Identification of dynamic panel logit models with fixed effects.  arXiv:2104.04590 (econ)
%Dicwer, L. and Zhao, S. (2014). Nonparametric empirical Bayes and maximum likelihood
%estimation for high-dimensional data analysis. ArXiv preprint arXiv:1
%\item
%Efron, B. (2019). Bayes, Oracle Bayes, and empirical Bayes. {\it Stat. Science}, {\bf 34},
%Number 2, 177-201.
%Greenshtein, E. and Ritov, Y. (2004). Persistence in high-dimensional linear predictor selection and the virtue of overparametrization. {\it Bernoulli}, Vol 10, Number 6 (2004), 971-988.
%\item
%Greenshtein, E. (2006). Best subset selection, persistence in high-dimensional statistical learning and optimization under l1 constraint. {\it Ann.Stat.}
%Volume 34, Number 5, 2367-2386.
\item
Efron, B. (2012). Large scale inference: empirical Bayes  methods for estimation , testing and prediction. Cambridge University Press.
\item 
Efron, B. (2014). Two modeling strategies for empirical Bayes estimation. Statistical science: a
review journal of the Institute of Mathematical Statistics, 29(2):285.
\item
Efron, B. and Thisted, R. (1976). Estimating the number of unseen species: How many words did Shakespeare know? Biometrika, 63(3), 435-447.
\item 
Good, I. J. and Toulmin, G. H. (1956). The number of new species, and the increase in population coverage, when a sample is increased. Biometrika, 43(1-2), 45-63.
\item
Greenshtein, E. and Itskov, T (2018), Application of Nonarametric  Empirical
Bayes to Treatment of Non-Response. {\it Statistica Sinica} 28 (2018), 2189-2208.
\item
Greenshtein, E. and Ritov, Y. (2022). Generalized maximum likelihood estimation of the mean of parameters of mixtures. with applications to sampling and to observational studies. {\it EJS} 16(2): 5934-5954.

\item Ho, K., and Rosen, A. M. (2017). Partial identification in applied research: benefits and challenges. In Advances in economics and econometrics: eleventh world congress (Vol. 2, pp. 307-359). Econometric Society Monographs.
    
%Gu, J. and Koenwer, R. (2017). Unobserved Heterogeneity in Income Dynamics: An Empirical Bayes Perspective. {\it Journal of Business and Economics Statistics}. Volume 35, 2017 - Issue 1
%\item
%Heinrich, Philippe, and Jonas Kahn. ”Strong identifiability and optimal minimax rates for finite mixture
%estimation.” Annals of Statistics 46.6A (2018): 2844-2870.
%\item
 %Ignatiadis Nikolaos and Wager Stefan (2021). Confidence Intervals for Nonparametric Empirical Bayes Analysis. To appear in JASA.
%\item
%Jiang, W. and Zhang, C.-H., (2009), General maximum likelihood empirical Bayes estimation of normal means.
%{\it Ann.Stat.}. {\bf 37} No. 4, 1647-1684.
%\item
%Karlis, D. and Xekalaki, E. (2005), Mixed Poisson Distributions.
%{\it International Statistical Review}, 73, 1, 35–58, Printed in Wales by Cambrian Printers.
\item
Kiefer, J. and Wolfowitz, J. (1956). Consistency of the maximum likelihood estimator in the presence of infinitely many incidental parameters. {\it Ann.Math.Stat.}
27 No. 4, 887-906.
\item  Koenker, R. and Gu, J. (2026). Empirical Bayes: Some Tools, Rules, and Duals. CAMBRIDGE University Press.
\item
Koenker, R. and Mizera, I. (2014). Convex optimization, shape constraints,
compound decisions and empirical Bayes rules. {\it JASA } 109, 674-685.
\item
Laird, N. (1978). Nonparametric maximum likelihood estimation of a mixing distribution. {\it JASA}  78, No 364, 805-811.
\item
Lindsay, B. G. (1995). Mixture Models: Theory, Geometry and Applications.
Hayward, CA, IMS.
\item
Little, R.J.A and Rubin, D.B. (2002). Statistical Analysis with Missing Data.
New York: Wiley
%\item
%Long, Feng and Lee, H. Dicwer (2018).
%Approximate nonparametric maximum likelihood for mixture models: A convex optimization approach to fitting arbitrary multivariate mixing distributions. {\it Computational Statistics \& Data Analysis} 122: 80-91.
%\item
%Robins, J. M. and  Ritov, Y (1997). Toward a curse of dimensionality appropriated
%(CODA) asymptotic theory for semi-parametric models. {\it Statistics in Medicine},
%16:285–319.
%\item
%Rosenbaum, P. R. and Rubin, D. B. (1983). The central role of the propensity score in observational studies for causal effects. {\it Biometrika,} Vol 70, Issue 1, April 1983, Pages 41–55.
%\item
%Teicher, H. (1963). Identifiability of finite mixtures. {\it Ann. Math. Stat.},
%{\bf 34}, No. 4, 1265-1269.
%\item
%Wolter, K. M. (1985). Introduction to Variance estimation. Second edition, Springer.
%\item
%J.Pfanzagl (1993). Incidental Versus Random Nuisance Parameters. {\it Ann Stat}, 21, 1663-1691.
%\item
%Sujayam Saha. Adityanand Guntuboyina. "
%``On the nonparametric maximum likelihood estimator for Gaussian location mixture densities with application to Gaussian denoising.'' Ann. Statist. 48 (2) 738 - 762, April 2020. https://doi.org/10.1214/19-AOS1817
\item
Robbins, H. (1951). Asymptotically subminimax solutions of compound decision problems. In Proceedings of the Second Berweley Symposium on Mathematical Statistics and
Probability, 1950 131–148. Univ. California, Berweley. MR0044803
\item
Robbins, H. (1956). An empirical Bayes approach to statistics. In Proc. Third Berweley
Symp. 157–164. Univ. California Press, Berweley. \\MR0084919
\item
Robbins, H. (1964). The empirical Bayes approach to statistical decision problems. Ann.
Math. Statist. 35 1–20. MR0163407
\item
Robbins, H. and Zhang, C-H. (1991). Estimating multiplicative  treatment effect under biased allocation. Biometrika 78(2) 349-354.
%Zhang, C.-H. (1997). Empirical Bayes and compound estimation of a normal mean.
%statist. Sinica 7 181–193.
%\item
%Zhang, C-H. (2005). Estimation of sums of random variables: Examples and information bounds. {\it Ann. Stat.} {\bf 33}, No.5. 2022-2041.
\item
Vardi, Y. (1985). Empirical distribution in selection bias  models.  Ann.Stat. 13  (1), 178-203.
\item
Turnbull, B.  W. (1976). The empirical distribution function with arbitrarily grouped, censored and truncated data.
jrssb (38) 3, 290-295. 
\item
Zhang, C-H. (2003)  Compound decision theory  and empirical Bayes  methods. Ann. Stat.   31 (2),  379-390.

\end{list}

\end{document}